\documentclass[10pt,draft,reqno]{amsart}
     \makeatletter
     \def\section{\@startsection{section}{1}%
     \z@{.7\linespacing\@plus\linespacing}{.5\linespacing}%
     {\bfseries
     \centering
     }}
     \def\@secnumfont{\bfseries}
     \makeatother
\newtheorem{theorem}{Theorem}[section]
\newtheorem{lemma}[theorem]{Lemma}

\newtheorem{corollary}[theorem]{Corollary}
\theoremstyle{definition}
\newtheorem{assumption}[theorem]{Assumption}
\newtheorem{definition}[theorem]{Definition}

\theoremstyle{remark}
\newtheorem{remark}[theorem]{Remark}
\newtheorem{hypothesis}[theorem]{Hypothesis}
\numberwithin{equation}{section} 

\newcommand{\Tr}{\mathop{\mathrm{Tr}}}
\renewcommand{\d}{\/\mathrm{d}\/}
\def\s{^{\star}}

\def\u{u^{n, \varepsilon}}
\def\ue{u^{\varepsilon}}
\def\uv{u^{\varepsilon}_{\theta}}
\def\n{u_{\theta_n}}
\def\m{w_{\theta_n}}
\def\uve{u^{\varepsilon}_{{\theta}^{\varepsilon}}}
\def\w{w^{\varepsilon}_{{\theta}^{\varepsilon}}}

\def\e{\varepsilon}

\def\t{t\wedge\tau_N}
\def\T{T\wedge\tau_N}

\begin{document}

\title[Large deviations for the shell model of turbulence]
{Large deviations for the shell model of turbulence perturbed by
L\'{e}vy noise}

\author[U. Manna]{Utpal Manna} \author[M. T. Mohan]{Manil .T.
Mohan}\address{School of Mathematics, Indian Institute of Science
Education and Research, Thiruvananthapuram, India}
\email{manna.utpal@iisertvm.ac.in, manil@iisertvm.ac.in}

\keywords{Large deviations, L\'{e}vy processes, GOY model.}

\subjclass[2000]{Primary 60F10; Secondary 60G51, 60H15, 76D06.}

\begin{abstract}
The Laplace principle for the strong solution of the stochastic
shell model of turbulence perturbed by L\'{e}vy noise is established
in a suitable Polish space using weak convergence approach. The
large deviation principle is proved using the well known results of
Varadhan and Bryc.

\end{abstract}

 \maketitle\setcounter{equation}{0}

\maketitle
\section{Introduction}
The large deviations theory is among the most classical areas in
probability theory with many deep developments and applications.
Several authors have established the Wentzell-Freidlin type large
deviation estimates for a class of infinite dimensional stochastic
differential equations (see for eg., Budhiraja and Dupuis
\cite{BD1}, Da Prato and Zabczyk \cite{DaZ}, Kallianpur and Xiong
\cite{KX}). In these works the proofs of large deviation principle
(LDP) usually rely on first approximating the original problem by
 time-discretization so that LDP can be shown for the resulting
 simpler problems via contraction principle, and then showing that
 LDP holds in the limit. The discretization method to establish LDP
  was introduced by Wentzell and Freidlin\cite{FW}. Dupuis and Ellis
   \cite{DE} have combined weak convergence methods to the stochastic
   control approach developed earlier by Fleming \cite{Fl} to the large deviations theory.

The literature associated to the LDP of stochastic partial
differential equations with L\'{e}vy noises is very few. De Acosta
\cite{Ac1, Ac} first studied the large deviations for L\'{e}vy
processes on Banach spaces and large deviations for solutions of
stochastic differential equations driven by Poisson measures.
Recently Budhiraja, Dupuis, and Maroulas \cite{BD2} and Maroulas
\cite{Mv}, using the theorems of Varadhan and Bryc \cite{DZ}, have
extended the result of Budhiraja and Dupuis \cite{BD1} to prove the
LDP for stochastic differential equations with Poisson noises by
first establishing the Laplace principles in Polish spaces using the
weak convergence approach. The other notable recent work is due to
Swiech and Zabczyk \cite{SZ}, where the large deviation principle
for solutions of abstract stochastic evolution equations perturbed
by small L\'{e}vy noise is proved using the theorems of Varadhan and
Bryc coupled with the techniques of Feng and Kurtz \cite{FK},
viscosity solutions of Hamilton-Jacobi-Bellman equations in Hilbert
spaces and control theory.

To the best of our knowledge, the only work available in the area of
LDP for the fluid dynamics models with jump processes is due to Xu
and Zhang \cite{XZ}, where they applied the theory of De Acosta to
establish the LDP for $2$-D Navier-Stokes equations with additive
L\'{e}vy noises.

This work deals with an infinite dimensional shell model, a
mathematical turbulence model, that received increasing attention in
recent years. Apparently there are only a few rigorous works on
infinite dimensional shell model, namely Constantin, Levant and Titi
\cite{CLT}, and Barbato, Barsanti, Bessaih and Flandoli \cite{Ba}
one in the deterministic case and the other in the stochastic case
with additive noise respectively. In both of these works a
variational semigroup formulation has been introduced. The work by
Manna, Sritharan and Sundar \cite{Ma3} deals with the existence and
uniqueness of the strong solutions of the stochastic shell model of
turbulence perturbed by multiplicative noise. They have also
established a LDP for the solution of the shell model by using weak
convergence approach developed on the theory by Budhiraja and Dupuis
\cite{BD1}. The LDP for the inviscid shell models has been proved by
Bessaih and Millet \cite{BM}. Recently Manna and Mohan \cite{Ma4}
has proved the existence and uniqueness of the strong solutions of
the shell model of turbulence perturbed by L\'{e}vy noise.

In this work, the authors established the LDP for the shell model of
turbulence with L\'{e}vy noise by proving the Laplace principle for
the strong solution in certain Polish space using the weak
convergence approach developed by Budhiraja, Dupuis and Maroulas
\cite{BD2}, and Maroulas \cite{Mv}  and finally applying the well
known results by Varadhan and Bryc.

The main result of this paper is as follows:
\begin{theorem}[Main Theorem]\label{MTH}
Let the stochastic shell model of turbulence perturbed by L\'{e}vy
Noise described by
\begin{eqnarray}
\d \ue + \big[\nu A\ue + B(\ue)\big] \d t &= &f(t) \d t + \sqrt{\e}\sigma(t, \ue) \d W(t)+\e \int_Zg(\ue,z)\tilde{N}(dt,dz)\nonumber\\
 \ue(0) &=& \xi,
\end{eqnarray}
has a unique strong solution in the Polish space
$X=\mathcal{D}([0,T];H)\cap \mathrm{L}^2(0,T;V).$ Let the solution
be denoted by $\ue=\mathcal{G}^{\e}\left(\sqrt{\e}W(\cdot),\e
N^{\e^{-1}}\right)$. Then the family $\{\ue :\e >0\}$ satisfies
Large Deviation Principle in $X$ with the rate function $I$ given by
\begin{align*}
I(\zeta)=\inf_{(\psi,\phi)\in
\mathbb{S}_\zeta}\left\{\int_0^T\int_Z\ell(\phi(t,z))\lambda(dz)\d
t+\frac{1}{2}\int_0^T\|\psi(s)\|_0^2ds\right\}.
\end{align*}
\end{theorem}

The construction of the paper is as follows. The next section is
devoted to the formulation the abstract stochastic GOY model, the
energy estimates, the existence and uniqueness of strong solutions.
Proofs have been omitted as these results are already been proved by
the authors \cite{Ma4}. In the last Section proof of the main
theorem has been given in a systematic way.

\section{The Stochastic GOY Model of Turbulence}
\setcounter{equation}{0}
\subsection{Preliminaries}
Let $(\Omega, \mathcal{F}, \mathbb{P})$ be a probability space
equipped with an increasing  family  $\{\mathcal{F}_t\}_{0 \leq t
\leq T}$ of sub-sigma-fields of $\mathcal{F}$  satisfying the usual
conditions of right continuity and $\mathbb{P}$-completeness. Let
$H$ be a real separable Hilbert space and $Q$ be a strictly
positive, symmetric, trace class operator on $H$.
\begin{definition}
A stochastic process $\{W(t)\}_{0\leq t\leq T}$ is said to be an
$H$-valued $\mathcal{F}_t$-adapted Wiener process with covariance
operator $Q$ if
\begin{itemize}
\item For each non-zero $h\in H$, $|Q^{1/2}h|^{-1} (W(t), h)$ is a standard one dimensional Wiener process,
\item For any $h\in H, (W(t), h)$ is a martingale adapted to $\mathcal{F}_t$.
\end{itemize}
\end{definition}
If $W$ is a an $H$-valued Wiener process with covariance operator
$Q$ with $\Tr Q < \infty$, then $W$ is a Gaussian process on $H$ and
$ \mathbb{E}(W(t)) = 0,\quad \text{Cov}\ (W(t)) = tQ, \quad t\geq
0.$ Let $H_0 = Q^{1/2}H.$ Then $H_0$ is a Hilbert space equipped
with the inner product $(\cdot, \cdot)_0$, $(u, v)_0 = (Q^{-1/2}u,
Q^{-1/2}v),\ \forall u, v\in H_0,$ where $Q^{-1/2}$ is the
pseudo-inverse of $Q^{1/2}$. Since $Q$ is a trace class operator,
the imbedding of $H_0$ in $H$ is Hilbert-Schmidt.

Let $L_Q$ denote the space of linear operators $S$ such that $S
Q^{1/2}$ is a Hilbert-Schmidt operator from $H$ to $H$. Define the
norm on the space $\mathrm{L}_Q$ by $|S|_{\mathrm{L}_Q}^2 =
\Tr(SQS^*)$.

\begin{definition}
A c\`{a}dl\`{a}g adapted process is called a L\'{e}vy process if it
has stationary independent increments and is stochastically
continuous.
\end{definition}The jump of $X_t$ at $t\geq 0$ is given by $\triangle X_t = X_t -
X_{t-}$. Let $Z\in \mathcal{B}(H)$, define $$ N(t, Z)= N(t, Z,
\omega)= \sum_{s: 0<s\leq t} \chi_{_Z} (\triangle X_s).$$

In other words, $N(t, Z)$ is the number of jumps of size $\triangle
X_s\in Z$ which occur before or at time $t$. $N(t, Z)$ is called the
\emph{Poisson random measure} (or \emph{jump measure}) of
$(X_t)_{t\geq 0}$. The differential form of this measure is written
as $N(dt, dz)(\omega)$.

We call $\tilde{N}(dt, dz) = N(dt, dz) - \lambda(dz)dt $ a
\emph{compensated Poisson random measure (cPrm)}, where
$\lambda(dz)dt $ is known as \emph{compensator} of the L\'{e}vy
process $(X_t)_{t\geq 0}$. Here $dt$ denotes the Lesbesgue measure
on $\mathcal{B}(\mathbb{R}^{+})$, and $\lambda(dz)$ is a
$\sigma$-finite measure on $(Z, \mathcal{B}(Z))$.
\begin{definition}
Let $H$ and $F$ be separable Hilbert spaces. Let $F_t:=
\mathcal{B}(H)\otimes\mathcal{F}_t$ be the product $\sigma$-algebra
generated by the semi-ring $ \mathcal{B}(H)\times\mathcal{F}_t$ of
the product sets $Z\times F,\quad Z\in\mathcal{B}(H),\quad F\in
\mathcal{F}_t$ ( where $\mathcal{F}_t$ is the filtration of the
additive process $(X_t)_{t\geq0})$. Let $T>0$, define
\begin{align*}
\mathbb{H}(Z) = & \big\{g : \mathbb{R}^+ \times Z \times \Omega
\rightarrow F,\textrm{ such that g is}\;\; F_T/\mathcal{B}(F) \;\;
measurable\;\;and \nonumber\\& \;\;\qquad g(t,z,\omega)\;\;is\;\;
\mathcal{F}_t - adapted\;\;\forall z\in Z, \forall t\in (0,T]\big\}.
\end{align*}
For $p\geq1$, let us define,
$$\mathbb{H}^{p}_\lambda([0,T]\times Z;F) = \left\{g\in \mathbb{H}(Z) :
 \int_0^T\int_Z\mathbb{E}[\|g(t,z,\omega)\|^p_F]\lambda(dz)dt < \infty \right\}.$$
\end{definition}
For more details see Mandrekar and R\"{u}diger \cite{MR}.

\begin{lemma}\label{bdg}
Let $1<p\leq 2$ and let $H$ be a separable Hilbert space of
martingale type $p$, i.e., there is a constant $K_p(H)>0$ such that
for all $H$-valued discrete martingale $\{M_n\}_{n=0}^N$ the
following inequality holds $$\sup_n\mathbb{E}|M_n|^p\leq
K_p(H)\sum_{n=0}^N\mathbb{E}|M_n-M_{n-1}|^p,$$ where set $M_{-1}=0$.
Assume that $g\in \mathbb{H}^p_{\lambda} ((0,\infty)\times Z;H)$.
Then there exists a constant $C=C_p(H)2^{2-p}$ only depending on $H$
and $p$ such that for $0<q\leq p$
$$\mathbb{E}\sup_{t\in
0<t\leq T}\left|\int_0^t\int_Zg(s,z,\omega)\tilde{N}(ds,dz)\right|^q
\leq
C\mathbb{E}\left(\int_0^T\int_Z|g(t,z,\omega)|^p\lambda(dz)dt\right)^{q/p}.$$
\end{lemma}
\noindent For proof see  Corollary C.2 of \cite{BH}.

\noindent \textbf{Note:} Let $C([0,T];H)$ and $\mathcal{D}([0,T];H)$
be the space of all continuous functions and the space of all
c\'{a}dl\'{a}g paths (right continuous functions with left limits)
from $[0,T]$ into $H$, where $H$ is a Hilbert space, endowed with
the uniform topology and Skorohod topology respectively.

\begin{lemma}\label{lemma100}(Kunita's Inequality)
Let us consider the stochastic differential equations driven by
L\'{e}vy noise of the form
\begin{eqnarray*}
\d u(t)=b(u(t))\d t+\sigma(t,u(t))\d
W(t)+\int_Zg(u(t-),z)\tilde{N}(dt,dz).
\end{eqnarray*}
Then for all $p\geq 2$, there exists $C(p,t)>0$ such that for each
$t>t_0\geq 0,$
\begin{eqnarray*}
&&\mathbb{E}\left[\sup_{t_0\leq s\leq
t}|u(s)|^p\right]\nonumber\\&&\leq
C(p,t)\left\{\mathbb{E}|u(0)|^p+\mathbb{E}\left(\int_{t_0}^t|b(u(r))|^p\d
r\right)+\mathbb{E}\left[\int_{t_0}^t|\sigma(r,u(r))|^p\d
r\right]\right.\\&&\left.+\mathbb{E}\left[\int_{t_0}^t(\int_Z|g(u(r-),z)|^2\lambda(dz))^{p/2}\d
r\right]+\mathbb{E}\left[\int_{t_0}^t\int_Z|g(u(r-),z)|^p\lambda(dz)dr\right]\right\}.
\end{eqnarray*}
\end{lemma}
\noindent For proof see Corollary 4.4.24 of \cite{Ap}.

\begin{remark}\label{young}
In this paper we will be frequently using the following form of
Young's inequality with exponents $p$ and $q$  $$ab\leq \e
a^p+C(\e)b^q,\;\;(a,b>0,\e>0)\;\;\textrm{ for }\;\;C(\e)=(\e
p)^{-q/p}q^{-1}.$$
\end{remark}

\subsection{The GOY Model of Turbulence}
\setcounter{equation}{0} The GOY model
(Gledger-Ohkitani-Yamada)~\cite{OY} is a particular case of so
called ``Shell model" (see, Frisch \cite{Fu}). This model is the
Navier-Stokes equation written in the Fourier space where the
interaction between different modes is preserved between nearest
modes. To be precise, the GOY model describes a one-dimensional
cascade of energies among an infinite sequence of complex
velocities, $\{u_n(t)\}$, on a one dimensional sequence of wave
numbers $k_n = k_0 2^n,\quad k_0 > 0,\ n=1, 2, \ldots,$ where the
discrete index $n$ is referred to as the ``shell index". The
equations of motion of the GOY model of turbulence have the form
\begin{align}\label{goy}
\frac{\d u_n}{\d t} + \nu k_n^2 u_n &+ i\big(a k_n
u\s_{n+1}u\s_{n+2} + b k_{n-1}u\s_{n-1}u\s_{n+1} + \nonumber\\ &
+ck_{n-2} u\s_{n-1}u\s_{n-2}\big) = f_n, \quad\text{for}\ n= 1, 2,
\ldots,
\end{align}
along with the boundary conditions $ u_{-1} = u_0 = 0.$ Here $u\s_n$
denotes the complex conjugate of $u_n$, $\nu > 0$ is the kinematic
viscosity and $f_n$ is the Fourier component of the forcing. $a, b$
and $c$ are real parameters such that energy conservation condition
$a + b + c =0$ holds (see Ohkitani and Yamada\cite{OY}).

\subsection{Functional Setting}
Let $H$ be a real Hilbert space such that
\begin{equation*}
H :=\left\{u=(u_1, u_2, \ldots) \in \mathbb{C}^{\infty}:
\sum_{n=1}^{\infty} |u_n|^2 < \infty\right\}.
\end{equation*}
For every $u, v \in H$, the scalar product $(\cdot,\cdot)$ and norm
$|\cdot|$ are defined on $H$ as $(u, v)_H = Re\ \sum_{n=1}^{\infty}
u_n v\s_n, \quad |u| = \left(\sum_{n=1}^{\infty}
|u_n|^2\right)^{1/2}.$ Let us now define the space
$$V :=\left\{u\in H: \sum_{n=1}^{\infty} k_n^2|u_n|^2 < \infty\right\},$$
which is a Hilbert space equipped with the norm $\|u\| =
\left(\sum_{n=1}^{\infty} k_n^2|u_n|^2\right)^{1/2}.$ The linear
operator $A: D(A) \rightarrow H$ is a positive definite, self
adjoint linear operator defined by
\begin{equation}\label{A}
Au=((Au)_1, (Au)_2, \ldots), \ \text{where}\ (Au)_n = k_n^2 u_n,
\quad\forall u\in D(A).
\end{equation}
The domain of $A$, $D(A) \subset H$, is a Hilbert space equipped
with the norm
$$\|u\|_{D(A)} = |Au| = \left(\sum_{n=1}^{\infty} k_n^4|u_n|^2\right)^{1/2}, \quad\forall u\in D(A).$$
Since the operator $A$ is positive definite, we can define the power
$A^{1/2}$ ,
$$A^{1/2}u = (k_1 u_1, k_2 u_2, \ldots), \quad\forall u=(u_1, u_2, \ldots).$$
Furthermore, we define the space
$$D(A^{1/2}) = \left\{u=(u_1, u_2, \ldots): \sum_{n=1}^{\infty} k_n^2 |u_n|^2 < \infty\right\},$$
which is a Hilbert space equipped with the scalar product
$$ (u, v)_{D(A^{1/2})} = (A^{1/2}u, A^{1/2}v), \quad\forall u, v\in D(A^{1/2}),$$
and the norm $$\|u\|_{D(A^{1/2})} = \left(\sum_{n=1}^{\infty} k_n^2
|u_n|^2\right)^{1/2}.$$ Note that $V = D(A^{1/2})$. We consider
$V^{\prime} = D(A^{-1/2})$ as the dual space of $V$. Then the
following inclusion holds
$$V\subset H = H^{\prime}\subset V^{\prime}.$$
We will now introduce the sequence spaces analogue to Sobolev
functional spaces. For $1\leq p <\infty$ and $s\in\mathbb{R}$,
$$\mathrm{W}^{s, p} :=\left\{u = (u_1, u_2, \ldots): \|A^{s/2}u\|_p =
\left(\sum_{n=1}^{\infty}(k_n^s|u_n|)^p\right)^{1/p} <
\infty\right\},$$ and for $p=\infty$, $$\mathrm{W}^{s, \infty}
:=\left\{u = (u_1, u_2, \ldots): \|A^{s/2}u\|_{\infty} = \sup_{1\leq
n<\infty} (k_n^s|u_n|) < \infty\right\},$$ where for
$u\in\mathrm{W}^{s, p}$ the norm is defined as
$\|u\|_{\mathrm{W}^{s, p}} = \|A^{s/2}u\|_p.$ Here $\|\cdot\|$
denotes the usual norm in the $l^p$ sequence space. It is clear from
the above definitions that $W^{1, 2} = V = D(A^{1/2})$.

\begin{remark}
For the shell model we can reasonably assume that the complex
velocities $u_n$ are such that $|u_n| <1$ for almost all $n$. Then
\begin{align*}
\| u \|_{l^4}^4 = \sum_{n=1}^{\infty} |u_n|^4 \leq
\left(\sum_{n=1}^{\infty} |u_n|^2\right)^2 = |u|^4,
\end{align*}
which leads to $H\subset l^4$.
\end{remark}

We now state a Lemma which is useful in this work. We omit the proof
since it is quite simple.
\begin{lemma}\label{La}
For any smooth function $u\in H$, the following holds:
\begin{align}
  \| u \|_{l^4}^4 \leq C |u|^2 \ \|u\|^2.\label{L4}
\end{align}
\end{lemma}

\subsection{Properties of the Linear and Nonlinear Operators}
We define the bilinear operator $B(\cdot, \cdot): V \times
H\rightarrow H$ as
$$B(u, v) = (B_1(u,v), B_2(u, v), \ldots),$$ where
\begin{align*}
B_n(u, v)= ik_n\left(\frac{1}{4}u\s_{n+1} v\s_{n-1} -
\frac{1}{2}(u\s_{n+1} v\s_{n+2} + u\s_{n+2} v\s_{n+1})
 + \frac{1}{8}u\s_{n-1} v\s_{n-2}\right).
\end{align*}
In other words, if $\{e_n\}_{n=1}^{\infty}$ be a orthonormal basis
of $H$, i.e. all the entries of $e_n$ are zero except at the place
$n$ it is equal to $1$, then
\begin{align}\label{B}
B(u, v)= i\sum_{n=1}^{\infty}k_n\left(\frac{1}{4}u\s_{n+1} v\s_{n-1}
- \frac{1}{2}(u\s_{n+1} v\s_{n+2} + u\s_{n+2} v\s_{n+1})
 + \frac{1}{8}u\s_{n-1} v\s_{n-2}\right)e_n.
\end{align}

 The following lemma says that $B(u, v)$ makes
sense as an element of $H$, whenever $u\in V$ and $v\in H$ or $u\in
H$ and $v\in V$. It also says that $B(u, v)$ makes sense as an
element of $V^{\prime}$. Here we state the following lemma which has
been proved in Constantin, Levant and Titi \cite{CLT} for the Sabra
shell model, but one can also prove the similar estimates for the
GOY model (see Barbato, Barsanti, Bessaih, and Flandoli\cite{Ba}).
\begin{lemma}\label{Bprop1}
(i) There exist constants $C_1 >0, C_2 >0$,
\begin{equation}
|B(u, v)| \leq C_1 \|u\| |v|, \quad\forall u\in V, v\in H,
\end{equation}
and
\begin{equation}
|B(u, v)| \leq C_2 |u| \|v\|, \quad\forall u\in H, v\in V.
\end{equation}
(ii) $B: H\times H\rightarrow V^{\prime}$ is a bounded bilinear
operator and for a constant $C_3 > 0$
\begin{equation}
\|B(u, v)\|_{V^{\prime}} \leq C_3 |u| |v|, \quad\forall u, v\in H.
\end{equation}
(iii) $B: H\times D(A)\rightarrow V$ is a bounded bilinear operator
and for a constant $C_4 > 0$
\begin{equation}
\|B(u, v)\|_{V} \leq C_4 |u| |Av|, \quad\forall u\in H, v\in D(A).
\end{equation}
(iv) For every $u\in V$ and $v\in H$
\begin{equation}\label{1}
(B(u, v), v) = 0.
\end{equation}
\end{lemma}

We now present one more important property of the nonlinear operator
$B$ in the following lemma which will play an important role in the
later part of this section and also in the next section. The proof
is straightforward and uses the bilinearity property of $B$.
\begin{lemma}\label{Bprop2}
If $w=u-v$, then
$$B(u, u)-B(v, v) = B(v, w) + B(w, v) + B(w, w).$$
\end{lemma}

With above functional setting and following the classical treatment
of the Navier-Stokes equation, one can write the stochastic GOY
model of turbulence \eqref{goy} with the L\'{e}vy forcing as the
following,
\begin{align}\label{goy1}
\d u + \big[\nu Au + B(u, u)\big] \d t = f(t) \d t +
\sqrt{\varepsilon}\sigma(t, u) \d W(t) +
\varepsilon\int_Zg(u,z)\tilde{N}(dt,dz)
\end{align}
$u(0) = u_0$, where $u \in H$, the operators $A$ and $B$ are defined
through \eqref{A} and \eqref{B} respectively,  $f=(f_1, f_2,
\ldots), \sigma(t, u)=(\sigma_1(t, u_1), \sigma_2(t, u_2), \ldots)$.
Here $(W(t)_{t\geq 0})$ is a $H$-valued Wiener process with trace
class covariance $Q$, and the space $L_Q$ has been defined in the
beginning of this section. Here $g(u,z)$ is a measurable mapping
from $H \times Z$ into H, where $Z$ is a measurable space and $Z\in
\mathcal{B}(H)$, and let $ \mathcal{D}([0,T];H) $ be the space of
all c\`{a}dl\`{a}g paths from $[0,T]$ into H.

Assume that $\sigma$ and $g$ satisfy the following hypotheses of
joint continuity, linear growth and Lipschitz condition:
\begin{hypothesis}\label{hyp}
The main hypothesis is the following,
\begin{itemize}
\item[H.1.]  The function $\sigma \in C([0, T] \times V; L_Q(H_0; H))$, and $g\in \mathbb{H}^2_{\lambda}([0, T] \times Z; H)$.

\item[H.2.]  For all $t \in (0, T)$, there exists a positive constant $K$ such that for all $u \in H$,
$$|\sigma(t, u)|^2_{L_Q} + \int_{Z} |g(u, z)|^2_{H}\lambda(dz) \leq K(1 +|u|^2).$$

\item[H.3.]  For all $t \in (0, T)$,  there exists a positive constant $L$ such that for all $u, v \in H$,
$$|\sigma(t, u) - \sigma(t, v)|^2_{L_Q} + \int_{Z} |g(u, z)-g(v, z)|^2_{H}\lambda(dz)\leq L|u - v|^2.$$
\end{itemize}
\end{hypothesis}

The following lemma shows that sum of the linear and nonlinear
operator is locally monotone in the $l^4$-ball.
\begin{lemma}\label{Mon}
For a given $r > 0$, let us denote by $\mathbb{B}_r$ the closed
$l^4$-ball in $V$: $\mathbb{B}_r = \big\{v\in V; \|v\|_{l^4} \leq
r\big\}.$ Define the nonlinear operator F on $V$ by $F(u):=-\nu Au -
B(u, u)$. Then for $0 < \varepsilon < \frac{\nu}{2L}$, where $L$ is
the positive constant that appears in the condition (H.3), the pair
$(F, \sqrt{\varepsilon}\sigma + \varepsilon\int_Zg(.,z)\lambda(dz))$
is monotone in $\mathbb{B}_r$, i.e. for any $u\in V$ and $v\in
\mathbb{B}_r$, and $w=u-v$,
\begin{align}\label{monotone}
&(F(u) - F(v), w)  - \frac{r^4}{\nu^3} |w|^2 +
\varepsilon\left[|\sigma(t,u)-\sigma(t,v)|^2_{L_Q}
\right.\nonumber\\&\left.\qquad\qquad\qquad\qquad\qquad+
\int_Z|g(u,z)-g(v,z)|^2\lambda(dz)\right]\leq0.
\end{align}
\end{lemma}
For proof see Lemma 3.6 of \cite{Ma4}.

\subsection{Energy Estimates and Existence Result}
 Let $H_n$ be defined as the $\text{span}\ \{e_1, e_2,
\cdots, e_n\},$ where $\{e_j\}$ is any fixed orthonormal basis in
$H$ with each $e_j \in D(A)$. Let $P_n$ denote the orthogonal
projection of $H$ to $H_n$. Define $u^n = P_n u$, to avoid confusion
with earlier notation $u_n$. Let $W_n = P_nW$ . Let $\sigma_n =
P_n\sigma $ and
$\int_Zg^n(\u(t-),z)\tilde{N}(dt,dz)=P_n\int_Zg(u(t-),z)\tilde{N}(dt,dz)$,
where $g^n=P_ng$. Define $\u$ as the solution of the following
stochastic differential equation in the variational form such that
for each $v \in H_n$,
\begin{align}\label{variational}
\d (\u(t) , v)& = (F(\u(t)), v)\d t + (f(t), v)\d t +
\sqrt{\varepsilon}(\sigma_n(t, \u(t)) \d W_n(t), v) \nonumber\\
&\quad + \varepsilon\int_Z\big(g^n(\u(t-),z),v\big)\tilde{N}(dt,dz),
\end{align}
with $u(0) = P_n u(0) $.

\begin{theorem}\label{energy}
With the above mathematical setting let $f$ be in $\mathrm{L}^2( [0,
T], V')$,  $u (0)$ be $\mathcal{F}_0$ measurable, $\sigma \in C([0,
T] \times V; L_Q(H_0; H))$,  $g\in \mathbb{H}^2_{\lambda}([0, T]
\times Z; H)$ and $\mathbb{E}|u (0)|^2 < \infty$. Let $\u$ denote
the unique strong solution of the stochastic differential equation
\eqref{variational} in $\mathcal{D}([0, T]; H_n)$. Then with $K$ as
in condition (H.2), the following estimates hold: \\For all
$\varepsilon$, and $0 \leq t \leq T$,
\begin{align}
&\mathbb{E}| \u(t)|^2 + \nu\int_0^t \mathbb{E}\| \u(s)\|^2 \d s
\nonumber\\ &\quad \leq \left(1 + \varepsilon KTe^{\varepsilon
KT}\right)\left(\mathbb{E}|u(0)|^2 +
\frac{1}{\nu}\int_0^t\|f(s)\|_{V^\prime}^2ds + \varepsilon KT\right)
,\label{energy1}
\end{align}
and for all $\varepsilon >0$,
\begin{align}
\mathbb{E}\left[\sup_{0\leq t\leq T} |\u(t)|^2 + \nu\int_0^T
\|\u(t)\|^2 \d t\right] \leq C\left(\mathbb{E}|u(0)|^2, \int_0^T
\|f(t)\|^2_{V^{\prime}} \d t, \nu, T\right)\label{energy2}.
\end{align}
\end{theorem}
\begin{assumption}\label{assum1}
Let $p\geq 2$, for all $t\in (0,T)$, and for all $u\in H$, there
exists a positive constant $K_1$ such that
\begin{eqnarray}\label{eqtn99}
|\sigma(s,u(t))|^p+\int_Z|g(u(t),z)|^p\lambda(dz)\leq
K_1\left(1+|u(t)|^p\right).
\end{eqnarray}
\end{assumption}

\begin{theorem}\label{thm99}
Let  $p\geq 2$, $u(0)$ be $\mathcal{F}_0$ measurable, $f\in
\mathrm{L}^p( [0, T], V')$, $\sigma \in C([0, T] \times V; L_Q(H_0;
H))$,  $g\in \mathbb{H}^p_{\lambda}([0, T] \times Z; H)$ and let
$\mathbb{E}|u(0)|^p< \infty$. Let $\u(t)$ denote the unique strong
solution to finite system of equations (\ref{variational}) in
$\mathcal{D}([0, T], H_n)$. Then with $K$ as in condition (H.2) and
$K_1$ as in Assumption, the following estimates hold:
\begin{eqnarray}\label{90}
\mathbb{E}\left[\sup_{0\leq t\leq T} |\u(t)|^p\right]
+\frac{p\nu}{2}\mathbb{E}\left(\int_0^T\|\u(t)\|^2|\u(t)|^{p-2}dt
\right)\nonumber\\\leq C\left(\mathbb{E}|u(0)|^2,K,K_1,p,
T,\nu,\int_0^T\|f(t)|^p_{V'}\d t\right).
\end{eqnarray}
\end{theorem}
\begin{proof}
Define
$\tau_N=\inf\left\{t:|\u(t)|^p+\int_0^t\|\u(s)\|^2|\u(s)|^{p-2}\d
s>N\right\}.$ Let us take the function $f(x)=|x|^p$ and apply the
It\^{o}'s lemma to the process $\u(t)$. Use the property of the
operators $A$ and $B$, apply Cauchy-Schwartz inequality and Young's
inequality (Remark \ref{young}) to the term
$|\u(s)|^{p-2}\left(f(s),\u(s)\right)$ to obtain,
\begin{align}\label{eqtn102}
&|\u(\t)|^p+\frac{p\nu}{2}\int_0^{\t}\|\u(s)\|^2|\u(s)|^{p-2}\d
s\nonumber\\&\leq
|u(0)|^p+\frac{p}{2\nu}\int_0^{\t}\|f(s)\|^p_{V'}\d
s+C_1(p,\nu)\int_0^{\t}|\u(s)|^p\d
s\nonumber\\&\;+p\int_0^{\t}|\u(s)|^{p-2}\left(\sigma(s,\u(s)),\u(s)\right)\d
W_n(s)\nonumber\\&\;+\frac{p(p-1)}{2}\int_0^{\t}|\u(s)|^{p-2}\textrm{Tr}(\sigma(s,\u(s))
Q\sigma(s,\u(s)))\d s\nonumber\\&\;+\int_0^{\t}\int_Z\left[|\u(s-)+g(\u(s-),z)|^p-|\u(s-)|^p\right]\tilde{N}(ds,dz)\nonumber\\
&\qquad+\int_0^{\t}\int_Z\left[|\u(s-)+g(\u(s-),z)|^p-|\u(s-)|^p\right.\nonumber\\
&\qquad\qquad\qquad\left.-p|\u(s-)|^{p-2}\left(g(\u(s-),z),\u(s-)\right)\right]\lambda(dz)ds,
\end{align}
where
$C_1(p,\nu)=\left(\frac{p-2}{2\nu}\right)\left(\frac{2}{p}\right)^{2/(p-2)}$.
Let us denote the last two terms on RHS by $I$. Consider
$\frac{p(p-1)}{2}\int_0^{\t}|\u(s)|^{p-2}\textrm{Tr}(\sigma(s,\u(s))
Q\sigma(s,\u(s)))\d s$ in (\ref{eqtn102}), apply Young's inequality
and $|\u|\leq \|\u\|$ to obtain,
\begin{align}\label{eqtn103}
&\frac{p(p-1)}{2}\int_0^{\t}|\u(s)|^{p-2}\textrm{Tr}(\sigma(s,\u(s))Q\sigma(s,\u(s)))\d
s\nonumber\\&\leq
\frac{p\nu}{4}\int_0^{\t}\|\u(s)\|^2|\u(s)|^{p-2}\d
s+C_2(p,\nu)\int_0^{\t}|\sigma(s,\u(s))|^p\d s,
\end{align}
where
$C_2(p,\nu)=(p-1)^{p/2}\left(\frac{2(p-2)}{p\nu}\right)^{\frac{p-2}{2}}$.
Now applying (\ref{eqtn103}) in (\ref{eqtn102}), taking supremum up
to time $\T$, and then taking the expectation, one can get,
\begin{eqnarray}\label{eqtn105}
&&\mathbb{E}\left[\sup_{0\leq t\leq
\T}|\u(t)|^p+\frac{p\nu}{4}\int_0^{\T}\|\u(t)\|^2|\u(t)|^{p-2}\d
t\right]\nonumber\\&&\leq\mathbb{E}\left[|u(0)|^p\right]+\frac{p}{2\nu}\int_0^{\t}\|f(s)\|^2_{V'}\d
s+C_1(p,\nu)\mathbb{E}\left[\int_0^{\t}\sup_{0\leq s\leq
t}|\u(s)|^p\d
s\right]\nonumber\\&&\quad+p\mathbb{E}\left[\sup_{0\leq t\leq
\T}\left|\int_0^{t}|\u(s)|^{p-2}\left(\sigma(s,\u(s)),\u(s)\right)\d
W_n(s)\right|\right]\nonumber\\&&\quad+C_2(p,\nu)
\mathbb{E}\left(\int_0^{\T}|\sigma(s,\u(s))|^p\d
s\right)+\mathbb{E}\left[\sup_{0\leq t\leq \T}I\right].
\end{eqnarray}
Let us consider  $p\mathbb{E}\left[\sup_{0\leq t\leq
\T}\left|\int_0^{t}|\u(s)|^{p-2}\left(\sigma(s,\u(s)),\u(s)\right)\d
W_n(s)\right|\right]$ from (\ref{eqtn105}) and apply
Burkholder-Davis-Gundy inequality, Young's inequality and
H\"{o}lder's inequality to get,
\begin{eqnarray}\label{eqtn108}
&&p\mathbb{E}\left[\sup_{0\leq t\leq
\T}\left|\int_0^t|\u(s)|^{p-2}\left(\sigma(s,\u(s),\u(s))\right)\d
W_n(s)\right|\right]\\&&\leq\frac{1}{2}\mathbb{E}\left[\sup_{0\leq
t\leq
\T}|\u(t)|^p\right]+(2(p-1))^{p-1}T^{\frac{p-2}{2}}\mathbb{E}\int_0^{\T}|\sigma(s,\u(s))|^p\d
s.\nonumber
\end{eqnarray}
Apply Kunita's inequality (see Lemma \ref{lemma100}) by taking $b=0$
and $\sigma=0$, we have,
\begin{eqnarray}\label{eqtn110}
\mathbb{E}\left[\sup_{0\leq t\leq \T}|I(t)|\right]&\leq&
C_3(p,T)\left\{\mathbb{E}\int_0^{\T}\left(\int_Z|g(\u(s-),z)|^2\lambda(dz)\right)^{p/2}\d
s\right.\nonumber\\&&\left.\quad+\mathbb{E}\left[\int_0^{\T}\int_Z|g(\u(s-),z)|^p\lambda(dz)\d
s\right]\right\}.
\end{eqnarray}
Thus we can write (\ref{eqtn105}) as,
\begin{eqnarray}\label{eqtn111}
&&\frac{1}{2}\mathbb{E}\left[\sup_{0\leq t\leq
\T}|\u(t)|^p\right]+\frac{p\nu}{4}\mathbb{E}\left(\int_0^{\T}\|\u(s)\|^2|\u(s)|^{p-2}\d
s\right)\nonumber\\&&\quad\leq\mathbb{E}\left[|u(0)|^p\right]+\frac{p}{2\nu}\int_0^{\t}\|f(s)\|^2_{V'}\d
s+C_1(p,\nu)\mathbb{E}\int_0^{\t}\sup_{0\leq s\leq t}|\u(s)|^p\d
s\nonumber\\&&\quad\quad+C_4(p,T,\nu)\mathbb{E}\left(\int_0^{\T}|\sigma(s,\u(s))|^p\d
s\right)\nonumber\\&&\quad\quad+C_3(p,T)\left\{\mathbb{E}\left[\int_0^{\T}\left(\int_Z|g(\u(s-),z)|^2\lambda(dz)\right)^{p/2}\d
s\right]\right.\nonumber\\&&\quad\quad\left.+\mathbb{E}\left[\int_0^{\T}\int_Z|g(\u(s-),z)|^p\lambda(dz)\d
s\right]\right\},
\end{eqnarray}
where $C_4(p,T,\nu)=C_2(p,\nu)+(2(p-1))^{p-1}T^{\frac{p-2}{2}}.$ Let
us take the last three terms of the above inequality and apply
Hypothesis \ref{hyp} and Assumption \ref{assum1} to get,
\begin{eqnarray}\label{eqtn114}
&&\frac{1}{2}\mathbb{E}\left[\sup_{0\leq t\leq
\T}|\u(t)|^p\right]+\frac{p\nu}{4}\mathbb{E}\left(\int_0^{\T}\|\u(s)\|^2|\u(s)|^{p-2}\d
s\right)\nonumber\\&&\quad
\leq\mathbb{E}\left[|u(0)|^p\right]+\frac{p}{2\nu}\int_0^{\t}\|f(s)\|^2_{V'}\d
s+\left[C(K,K_1,p,\nu,T)\right]T\nonumber\\
&&\qquad+\left[C(K,K_1,p,\nu,T)\right]\mathbb{E}\left(\int_0^{\T}\sup_{0\leq
s\leq t}|\u(s)|^p\d t\right).
\end{eqnarray}
Note that $\T\rightarrow T$ a.s. as $N\rightarrow \infty$. Finally
taking the limit in the above estimate (\ref{eqtn114}) and apply
Gronwall's inequality to get the result.
\end{proof}

\begin{definition}
A strong solution $\ue$ of the stochastic GOY model is defined on a
given probability space $(\Omega, \mathcal{F}, \mathcal{F}_{t},
\mathbb{P})$ as a $\mathrm{L}^p\left(\Omega;\mathrm{L}^{\infty
}(0,T; H)\cap \mathcal{D}(0,T;
H)\right)\cap\mathrm{L}^2\left(\Omega;\mathrm{L}^2(0,T;V)\right)$
valued adapted process which satisfies
\begin{align}\label{15a}
&\d \ue + \big[\nu A\ue + B(\ue, \ue)\big] \d t = f(t) \d t +
\sqrt{\e}\sigma(t, \ue) \d W(t) +
\varepsilon\int_Zg(\ue,z)\tilde{N}(dt,dz)
\end{align}
$\ue(0) = u_0$, in the weak sense and also the energy inequality in
Theorem \ref{thm99}.
\end{definition}


\begin{theorem}\label{existence}
Let  $u (0)$ be $\mathcal{F}_0$ measurable and $\ \mathbb{E}|u_0|^2
< \infty.$ Let $f\in \mathrm{L}^p(0, T; V^{\prime}).$ We also assume
that $0 < \e < \frac{\nu}{L}$ and the diffusion coefficient
satisfies the conditions (H.1)-(H.3). Then there exists unique
adapted process $\ue(t, x, w)$ with the regularity
$$\ue \in \mathrm{L}^p\left(\Omega; \mathcal{D}(0, T; H)\right) \cap \mathrm{L}^2\left(\Omega;\mathrm{L}^2(0, T; V)\right)$$
satisfying the stochastic GOY model \eqref{15a} and the a priori
bounds in Theorem \ref{thm99}.
\end{theorem}
\begin{proof}
The theorem can be proved using the local monotonicity property and
the energy estimates. A version of the theorem with
$$\ue \in \mathrm{L}^2(\Omega; \mathcal{D}(0, T; H) \cap \mathrm{L}^2(0, T; V))$$
has been proved in Theorem 4.4 of \cite{Ma4}.
\end{proof}

\section{Large Deviation Principle}
\setcounter{equation}{0} In this section we first give an abstract
formulation and basic properties for a class of large deviation
problems and then prove the Main Theorem \ref{MTH}.

Let us denote by $X$ a complete separable metric space and
$\{\mathbb{P}_{\varepsilon}: \varepsilon
> 0\}$ a family of probability measures on the Borel subsets of $X$.

\begin{definition}
A function $I : X\rightarrow [0, \infty]$ is called a \emph{rate
function} if $I$ is lower semicontinuous. A rate function $I$ is
called a \emph{good rate function} if for arbitrary $M \in [0,
\infty)$, the level set $K_M = \{x: I(x)\leq M\}$ is compact in $X$.
\end{definition}
\begin{definition}\label{LDP}(Large Deviation Principle)
 We say that a family of probability measures $\{\mathbb{P}_{\varepsilon}: \varepsilon > 0\}$
 satisfies the \emph{large deviation principle} (LDP) with a good rate
 function $I$ satisfying,
\begin{enumerate}
\item[(i)] for each closed set $F\subset X$
$$ \limsup_{\varepsilon\rightarrow 0}\ \varepsilon\ \log \mathbb{P}_{\varepsilon}(F) \leq -\inf_{x\in F} I(x),$$
\item[(ii)] for each open set $G\subset X$
$$ \liminf_{\varepsilon\rightarrow 0}\ \varepsilon\ \log \mathbb{P}_{\varepsilon}(G) \geq -\inf_{x\in G} I(x).$$
\end{enumerate}
\end{definition}

\begin{lemma}\label{VL}(Varadhan's Lemma \cite{Va})
Let $E$ be a Polish space and $\{X^{\varepsilon}: \varepsilon > 0\}$
be a family of $E$-valued random elements satisfying LDP with rate
function $I$. Then $\{X^{\varepsilon}: \varepsilon > 0\}$ satisfies
the Laplace principle on $E$ with the same rate function $I$ if for
all $h \in C_{b}(E)$,
\begin{equation}
\lim_{\varepsilon \rightarrow 0} {\varepsilon }\log
\mathbb{E}\left\{\exp\left[-
\frac{1}{\varepsilon}h(X^{\varepsilon})\right]\right\} = -\inf_{x
\in E} \{h(x) + I(x)\}. \label{LP}
\end{equation}
\end{lemma}
\begin{lemma}\label{BL}(Bryc's Lemma \cite{DZ})
The Laplace principle implies the LDP with the same rate function.
More precisely, if $\{X^{\varepsilon}: \varepsilon > 0\}$ satisfies
the Laplace principle on the Polish space $E$ with the rate function
$I$ and if the limit
\begin{equation*}
\lim_{\varepsilon \rightarrow 0} {\varepsilon }\log
\mathbb{E}\left\{\exp\left[-
\frac{1}{\varepsilon}h(X^{\varepsilon})\right]\right\} = -\inf_{x
\in E} \{h(x) + I(x)\}
\end{equation*}
is valid for all $h \in C_{b}(E)$, then $\{X^{\varepsilon}:
\varepsilon > 0\}$ satisfies the LDP on $E$ with rate function $I$.
\end{lemma}
Note that, Varadhan's Lemma together with Bryc's converse of
Varadhan's Lemma state that for Polish space valued random elements,
the Laplace principle and the large deviation principle are
equivalent.

We will now define the function spaces required for the formulation
of the large deviation problem. These spaces are defined based on
the theory developed in Budhiraja, Dupuis and Maroulas \cite{BD2}.

Let $X$ be a locally compact Polish space and let $X_T =[0,T]\times
X$ for any $T\in (0,\infty)$. Let $\mathcal{M}(X)$ be the space of
all measures $\mu$ on $(X,\mathcal{B}(X))$, satisfying
$\mu(K)<\infty$ for every compact subset $K$ of $X$. We endow
$\mathcal{M}(X)$ with the weakest topology such that for every $f\in
C_c(X)$ the function $\mu \rightarrow
\left<f,\mu\right>=\int_Xf(x)\mu(dx)$, $\mu\in \mathcal{M}(X)$ is a
continuous function. This topology can be metrized such that
$\mathcal{M}(X)$ is a Polish space. Let
$\mathbb{M}=\mathcal{M}(X_T)$ and let $\mathbb{P}$ be the unique
probability measure on $(\mathbb{M},\mathcal{B}(\mathbb{M}))$. Then
$\mathcal{B}(\mathbb{M})$ will denote a Borel $\sigma$-field on the
space $\mathcal{M}(X)$. For more details see Budhiraja, Dupuis and
Maroulas \cite{BD2}. Let us denote the product space
$C([0,T];H)\times \mathbb{M}$ by $\mathbb{V}(H)$. Define
$\mathcal{G}_t=\sigma \big\{N(s,Z), W(s):0\leq s\leq t, Z\in
\mathcal{B}(X_T)\big\}.$

For $\theta >0$, define $\mathbb{P}_\theta$ the unique probability
measure on $(\mathbb{V}(H),\mathcal{B}(\mathbb{V}(H)))$ such that
under $\mathbb{P}_\theta$
\begin{itemize}
\item [(i)] $W(t)$ is an $H$-valued $Q$-Wiener process.
\item [(ii)] $N$ is a Poisson Random Measure with intensity
measure $\lambda_T$.
\item [(iii)] $\{W(t),t\in [0,T]\}$, $\{N(t,Z), t\in
[0,T]\}$ are $\mathcal{G}_t$ martingales for every $Z\in
\mathcal{B}(X_T)$.
\end{itemize}

Now let us consider the $\mathbb{P}$-completion of the filtration
$\{\mathcal{G}_t\}$ and denote it by $\{\mathcal{F}_t\}$. We denote
by $\mathcal{P}$ the predictable $\sigma$-field on $[0,T]\times
\mathbb{V}(H)$ with the filtration $\{\mathcal{F}_t :0\leq t\leq
T\}$ on $(\mathbb{V}(H),\mathcal{B}(\mathbb{V}(H)))$. Let
$\mathcal{A}$ be the class of all $(\mathcal{P}\otimes
\mathcal{B}(X))/\mathcal{B}[0,\infty)$ measurable maps $\phi :
X_T\times \mathbb{V}\rightarrow [0,\infty)$. For $\phi \in
\mathcal{A}$, define the counting process $N^\phi$ on $X_T$ as
follows,
$$N^\phi\left(t,Z\right)=\int_{(0,t]\times Z}\int_0^\infty 1_{[0,\phi(s,z)]}(r)\tilde{N}(ds,dz)\d r,\quad t\in
[0,T],\quad Z\in \mathcal{B}(X).$$ Here $N^\phi$ is to be thought of
as a controlled random measure, with $\phi$ selecting the intensity
for the points of location $z$ and time $s$, in a possibly random
but non-anticipating way. Let us define $\ell:[0,\infty)\rightarrow
[0,\infty)$ by $\ell(r)=r\log r-r+1,\quad r\in [0,\infty).$ For any
$\phi \in \mathcal{A}$, let us define $L_T(\phi)$ by
$$L_T(\phi)=\int_0^T\int_Z
\ell(\phi(t,z,\omega))\lambda(dz)dt.$$ Define $\mathcal{P}_2 \equiv
\left\{\psi :\psi \;\mathrm{is}\;
\mathcal{P}/\mathcal{B}(\mathbb{R}) \; \mathrm{measurable}\;
\mathrm{and}\; \int_0^T\|\psi(s)\|_0^2 ds <\infty\; \mathbb{P}-
\mathrm{a.s}\; \right\}$ and let us set
$\mathcal{U}(H)=\mathcal{P}_2\times \mathcal{A}.$ For $\psi\in
\mathcal{P}_2$ let us define $\tilde{L}(\psi)$ by
$$\tilde{L}_T(\psi) =\frac{1}{2}\int_0^T \|\psi(s)\|_0^2 ds,$$ and for
$u=(\psi,\phi)\in \mathcal{U},$ set
$\bar{L}_T(u)=L_T(\phi)+\tilde{L}_T(\psi).$ For $\psi\in
\mathcal{P}_2$, let $W^\psi$ be
$$W^\psi(t)=W(t)+\int_0^t\psi(s)ds,\;t\in [0,T].$$
Let us define for $N\in\mathbb{N}$, $\tilde{S}^N(H_0) =
\left\{\psi\in \mathrm{L}^2([0,T]:H_0): \tilde{L}_T(\psi)\leq
N\right\}.$ Also let us define $S^N=\left\{\phi:X_T\rightarrow
[0,\infty) :L_T(\phi)\leq N\right\}.$ The convergence in
$\mathbb{M}$ is essentially equivalent to weak convergence on
compact subsets. The super linear growth of $\ell$ implies that
$\left\{\lambda_T^g:g\in S^N\right\}$ is a compact subset of
$\mathbb{M}$ where $$\lambda_T^g =\int_0^T\int_Z g(s,z)\lambda(dz)
\d s,\; Z\in \mathcal{B}(X_T).$$ Throughout we consider the topology
on $S^N$  obtained through this identification which makes $S^N$ a
compact space. We set $\bar{S}^N=\tilde{S}^N(H_0)\times S^N$ with
the usual product topology. Also let
$\mathcal{U}=\mathcal{P}_2(H)\times \mathcal{A}$ and set
$\mathbb{S}=\bigcup_{N\geq 1} \bar{S}^N$ and let $\mathcal{U}^N$ be
the space of all $\bar{S}^N$- valued controls, $\mathcal{U}^N
=\left\{u=(\psi,\phi)\in \mathcal{U}:u(\omega)\in \bar{S}^N,\;
\mathbb{P}\;\mathrm{a.e}\; \omega\right\}.$

Let $X$ and $X_0$ denote Polish spaces and for $\varepsilon>0$ let
$\mathcal{G}^\varepsilon : X_0\times \mathbb{V}(H)\rightarrow X$ be
a measurable map. Define
$$u^\varepsilon =\mathcal{G}^\varepsilon \left(\sqrt{\varepsilon}W(\cdot),\varepsilon N^{\varepsilon^{-1}}\right).$$
We are interested in the large deviation principle for $u^{\e}$ as
$\e \to 0$.

\begin{assumption}\label{assum}
There exists a measurable map $\mathcal{G}^0:X_0\times
\mathbb{V}(H)\rightarrow  X$ such that the following hold:
\begin{itemize}
\item [(i)] Let $\left\{\theta^\e =(\psi^\e,\phi^\e) \in \mathcal{U},
\theta ^\e (\omega)\in \bar{S}^M, \mathbb{P}-\;\mathrm{a.e.}\;
\omega\right\}\subset \mathcal{U}^M$ for $M<\infty$,  $\theta^\e$
converges in distribution on $\bar{S}^M$-valued random elements
$\theta =(\psi,\phi)$. Then $$\mathcal{G}^\e
\left(\sqrt{\e}W(\cdot)+\int_0^\cdot\psi^\e (s)ds,\e
N^{\e^{-1}\phi^{\e}}\right)\longrightarrow
\mathcal{G}^0\left(\int_0^\cdot\psi(s)ds,\lambda_T^\phi\right).$$
\item [(ii)] For every $M < \infty$, the set
$$K_M = \left\{\mathcal{G}^0\left(\int_0^\cdot\psi(s)ds,\lambda_T^\phi\right):(\phi,\psi)\in
\mathcal{U}^M\right\}$$ is a compact subset of $X$.
\end{itemize}
\end{assumption}

\noindent For each $\zeta \in X$, define $
\mathbb{S}_\zeta=\left\{(\psi,\phi)\in
\mathbb{S}:\zeta=\mathcal{G}^0\left(\int_0^\cdot\psi(s)ds,\lambda_T^\phi\right)\right\}.$
Let $I:X\rightarrow [0,\infty]$ be defined as
\begin{align}\label{rate}
I(\zeta) =\inf_{q=(\psi,\phi)\in
\mathbb{S}_\zeta}\left\{\int_0^T\int_Z\ell(\phi(t,z))\lambda(dz)\d
t+\frac{1}{2}\int_0^T\|\psi(s)\|_0^2ds\right\},
\end{align}
where infimum over an empty set is taken as $\infty$. Also here
$Z\subset \mathcal{B}( X)$.

We now state an important result by Budhiraja, Dupuis and Maroulas
\cite{BD2} (see Theorem 4.2 of \cite{BD2}).
\begin{theorem}\label{main}
Let $u^\varepsilon =\mathcal{G}^\varepsilon
\left(\sqrt{\varepsilon}W(\cdot),\varepsilon
N^{\varepsilon^{-1}}\right)$. If $\{\mathcal{G}^{\e}\}$ satisfies
the Assumption \ref{assum}, then the family $\{u^{\e}: \e > 0\}$
satisfies the Laplace principle in $X$ with rate function $I$ given
by \eqref{rate}.
\end{theorem}

\begin{remark}\label{rem1}
\item{1.} Notice that, since the underlying space $X$ is Polish,
the family $\{u^{\e}: \e > 0\}$ satisfies the LDP in $X$ with the
same rate function $I$.

\item{2.} Assumption \ref{assum}(i) is a statement on the weak convergence of a certain
 family of random variables and is at the core of weak convergence approach
  to the study of large deviations. Assumption \ref{assum}(ii) essentially says that the
  level sets of the rate function are compact.
\end{remark}

\begin{remark}
The stochastic GOY model in consideration,
\begin{align*}
&\d \ue + \big[\nu A\ue + B(\ue)\big] \d t = f(t) \d t +
\sqrt{\e}\sigma(t, \ue) \d W(t)+\e \int_Zg(\ue,z)\tilde{N}(dt,dz)
\end{align*}
$\ue(0) = \xi$ has a unique strong solution in the Polish space
$X=\mathcal{D}([0, T]; H) \cap \mathrm{L}^2(0, T; V)$. The solution
to the stochastic GOY model, denoted by $\ue$, can be written as
$\mathcal{G}^{\e}\left(\sqrt{\e}W(\cdot),\e N^{\e^{-1}}\right)$ for
a Borel measurable function $\mathcal{G}^{\e}: \mathcal{D}([0, T];
H)\rightarrow X$ (see Karatzas and Shreve \cite{KS}, page 310;
Vishik and Fursikov \cite{VF}, Chapter X, Corollary 4.2). For more
details about this formulation see Chapter IV (classical
Yamada-Watababe argument) of Ikeda and Watanabe \cite{IW} and
Section 3.2 of Budhiraja, Chen and Dupuis \cite{BCD}.
\end{remark}
The aim of this section is to verify that such a $\mathcal{G}^{\e}$
satisfies Assumption \ref{assum}. Then applying the Theorem
\ref{main} the LDP for $\{\ue : \e>0\}$ in $X$ can be established.

The LDP for $\{\ue : \e>0\}$ in $X$ have been proved here
systematically in four steps. In the first and second Theorems we
show the well posedness of certain controlled stochastic and
controlled deterministic equations in $X$. These results help to
prove the last two main Theorems on the compactness of the level
sets and weak convergence of the stochastic control equation stated
in Assumption \ref{assum}.

\begin{theorem}\label{scontrol}
For any $\theta \in \mathcal{U}^M$, $0 < M < \infty$, the stochastic
control equation
\begin{align}
&\d \uv(t) + \left[\nu A\uv(t) + B(\uv(t), \uv(t))\right]\d t
\nonumber\\&\qquad \qquad\qquad =\left[f(t) + \sigma(t,
\uv(t))\psi(t)+\int_Zg(\uv(t),z)\ell(\phi(t,z))\lambda(dz)\right]\d t\nonumber\\
&\qquad \qquad\qquad\quad+\sqrt{\e}\sigma(t, \uv(t)) \d W(t)+\e
\int_Zg(\uv(t),z)\tilde{N}(dt,dz),\label{11}
\end{align}
$\uv(0) = \xi\in H$ has a unique strong solution in
$\mathrm{L}^2(\Omega; X)$, where $X=\mathcal{D}(0, T; H) \cap
\mathrm{L}^2(0, T; V)$, $f \in \mathrm{L}^2(0, T; V^{\prime})$ and
$\sigma, \int_Zg(\cdot,z)\lambda(dz)$ will satisfy the hypotheses
H.1.--H.3. in Section 2.
\end{theorem}

\begin{proof}
Following the proof of Theorem \ref{thm99}, we can prove that if
$\uv (t)$ is a strong solution of the stochastic controlled equation
\eqref{11}, the following energy estimate holds:
\begin{align}\label{12}
\mathbb{E}\left(\sup_{0\leq t\leq T} |\uv (t)|^2 + \nu\int_0^T \|\uv
(t)\|^2 \d t\right) \leq C,
\end{align}
where $C = C\left(|\xi|^2, \int_0^T \|f\|^2_{V^{\prime}} \d t, \nu,
K, T, M\right)$ is a positive constant.

The proofs of the existence and uniqueness of the strong solution of
the stochastic controlled equation (\ref{11}) follow from Theorem
\ref{existence}. The proofs can be obtained from Manna and Mohan
\cite{Ma4} after a few modifications as needed due to the presence
of the control term. The energy estimate obtained in (\ref{12})
plays a crucial role in the proof of the existence and uniqueness
theorems.
\end{proof}

\begin{corollary}\label{re11}
Since $V\subset H$, $|u|\leq \|u\|$, from (\ref{12}), we have,
$$\mathbb{E}\left(\sup_{0\leq t\leq T} |\uv(t)|^2 +\nu\int_0^T|\uv(t)|^2\d t\right)\leq
C.$$
\end{corollary}

\begin{theorem}\label{dcontrol}
Let $\theta=(\psi,\phi)\in \mathcal{U}$ ; $f \in \mathrm{L}^2(0, T;
V^{\prime})$ and $\sigma$, $g$ satisfy the hypotheses H.1.--H.3. in
Section 2. Then the equation
\begin{align}\label{18}
\d u_\theta(t) + [\nu Au_\theta(t)+ B(u_\theta(t), u_\theta(t))]\d t
&= f(t)\d t + \sigma(t, u_\theta(t)) \psi(t) \d
t\nonumber\\&\quad+\int_Zg(u_\theta(t),z)\ell(\phi(t,z))\lambda(dz)\d
t,
\end{align}
where $u_\theta(0)=\xi\in H$, has a unique strong solution in
$X=\mathcal{D}(0, T; H) \cap \mathrm{L}^2(0, T; V)$.
\end{theorem}

\begin{proof}
This result can be considered as a particular case of the previous
Theorem \ref{scontrol}, where the noise term or the diffusion
co-efficient is absent.
\end{proof}

Next we state an important lemma from Budhiraja and Dupuis
\cite{BD1}.

\begin{lemma}\label{BDL1}
Let $\{\psi_n\}$ be a sequence of elements from $\tilde{S}^M$ for
some finite $M > 0$. Let $\psi_n$ converges in distribution to
$\psi$ with respect to the weak topology on $\mathrm{L}^2(0, T;
H_0)$. Then $\int_0^{\cdot}\psi_n(s)\d s$ converges in distribution
as $C(0, T; H)$- valued processes to $\int_0^{\cdot}\psi(s)\d s$ as
$n \to \infty$.
\end{lemma}

Now we are ready to check the Assumptions \ref{assum}.
\begin{theorem}[Compactness]\label{compact}
Let $M < \infty$ be a fixed positive number and let $\xi\in H$ be
deterministic. Let
$$ K_M :=\left\{u_\theta \in \mathcal{D}(0, T; H) \cap \mathrm{L}^2(0, T; V); \theta\in
\mathcal{U}^M\right\},$$ where $u_\theta$ is the unique solution in
$X=\mathcal{D}(0, T; H) \cap \mathrm{L}^2(0, T; V)$ of the
deterministic controlled equation \eqref{18}, with $u_\theta(0) =
\xi\in H$. Then $K_M$ is compact in $X$.
\end{theorem}

\begin{proof}
Let us consider a sequence $\{\n\}$ in $K_M$, where $\n$ corresponds
to the solution of \eqref{18} with control $\theta_n \in
\mathcal{U}^M$ in place of $\theta$, i.e.
\begin{align}\label{19}
\d \n(t) &+ [\nu A\n(t) + B(\n(t), \n(t))]\d t=f(t)\d t + \sigma(t,
\n(t))\psi_n(t)\d t
\nonumber\\&\qquad\qquad\qquad\qquad\qquad\qquad+\int_Z
g(\n(t),z)\ell(\phi_n(t,z))\lambda(dz)\d t,
\end{align}
with $\n(0)=\xi\in H$. Then by weak compactness of $\mathcal{U}^M$,
there exists a subsequence of $\{\theta_n\}$, still denoted by
$\{\theta_n\}$, which converges weakly to $\theta\in \mathcal{U}^M$
in $\mathcal{U}$. We need to prove $\n \to u_\theta$ in $X$ as
$n\to\infty$, or in other words,
\begin{align}\label{20}
\sup_{0\leq t\leq T} |\n(t) - u_\theta(t)|^2 + \int_0^T \|\n(t) -
u_\theta(t)\|^2 \d t \longrightarrow 0,\;\;\textrm{ as
}\;\;n\to\infty.
\end{align}

According to the Theorem \ref{dcontrol}, $u_\theta$ is unique strong
solution in $X$ of the deterministic controlled equation \eqref{18}.
Hence it is obvious to note that, $u_\theta$ satisfies the following
a-priori estimate
\begin{align}\label{21}
\sup_{0\leq t\leq T} |u_\theta (t)|^2 + \int_0^T \|u_\theta (t)\|^2
\d t \leq C,
\end{align}
where $C = C\left(|\xi|^2, \int_0^T \|f\|^2_{V^{\prime}} \d t, \nu,
K, T, M\right)$ is a positive constant.

For the proof, we refer the Theorem \ref{scontrol}, where the
stochastic version of the above a priori estimate has been worked
out.

Let $\m = \n - u_\theta$. Then $\m$ satisfies the following
differential equation
\begin{align}\label{222}
&\d \m(t) +[\nu A\m(t)+B(\n(t), \n(t)) - B(u_\theta(t),
u_\theta(t))]\d
t\nonumber\\
&\quad = [\sigma(t, \n(t)) \psi_n(t) - \sigma(t, u_\theta(t))
\psi(t)]\d t\nonumber\\&\quad\quad+\int_Z
\left[g(\n(t),z)\ell(\phi_n(t,z))-g(u_\theta(t),z)\ell(\phi(t,z))\right]\lambda(dz)\d
t.
\end{align}
Let us multiply \eqref{222} by $\m(t)$ and then integrating from
$0\leq s \leq t$ to get,
\begin{align}
&|\m(t)|^2 + 2\nu\int_0^t\|\m(s)\|^2 \d s \nonumber\\&\quad+
2\int_0^t \big( B(\n(s), \n(s)) - B(u_\theta(s), u_\theta(s)),
\m(s)\big)\d s=I_1,\;\;\textrm{ where}\label{22} \\&I_1=2\int_0^t
\big(\sigma(s, \n(s)) \psi_n(s) - \sigma(s, u_\theta(s)) \psi(s),
\m(s)\big)\d s\nonumber\\&\quad+2\int_0^t\int_Z
\big(g(\n(s),z)\ell(\phi_n(s,z))-g(u_\theta(s),z)\ell(\phi(s,z)),\m(s)\big)\lambda(dz)\d
s.
\end{align}
But $B(\n, \n) - B(u_\theta, u_\theta)=B(u_\theta, \m) + B(\m,
u_\theta) + B(\m, \m)$, from Lemma \ref{Bprop2}. Using this, the
properties $(ii)$ and $(iv)$ of the bilinear operator $B$ given in
Lemma \ref{Bprop1} and using the inequality $2ab\leq \nu
a^2+\frac{1}{\nu}b^2$, one can obtain,
\begin{align}
 2\big| \big( B(\n(s)) - B(u_\theta(s)),
\m(s)\big)\big|\leq \nu\|\m(s)\|^2  + \frac{1}{\nu} |\m(s)|^2
|u_\theta(s)|^2.\label{23}
\end{align}
The term $I_1$ can be written as,
\begin{align}
|I_1| &\leq  2\int_0^t  \left|\big(\big(\sigma(s, \n(s)) - \sigma(s,
u_\theta(s))\big)\psi_n(s),
\m(s)\big)\right| \d s \nonumber\\
&\quad + 2\left|\int_0^t \big(\sigma(s, u_\theta(s)) (\psi_n(s)
-\psi(s)),\m(s)\big)\d
s\right|\nonumber\\&\quad+2\int_0^t\int_Z\left|\big((g(\n(s),z)-g(u_\theta(s),z))\ell(\phi_n(s,z)),\m(s)\big)\right|\lambda(dz)\d
s\nonumber\\&\quad+2\left|\int_0^t\int_Z\big(g(u_\theta(s),z)(\ell(\phi_n(s,z))-\ell(\phi(s,z))),\m(s)\big)\lambda(dz)\d s\right|\nonumber\\
 & \leq I_2+2\sup_{0\leq t\leq T}\left|\int_0^t \big(\sigma(s, u_\theta(s)) (\psi_n(s)
-\psi(s)),\m(s)\big)\d s\right|\label{24}\\&\quad+2\sup_{0\leq t\leq
T}\left|\int_0^t\int_Z\big(g(u_\theta(s),z)(\ell(\phi_n(s,z))-\ell(\phi(s,z))),\m(s)\big)\lambda(dz)\d
s\right|,\nonumber
\end{align}
\begin{align}
\textrm{where
}&I_2=2\int_0^t\big|\sigma(s,\n(s))-\sigma(s,u_\theta(s))\big|_{\mathrm{L}_Q}|\psi_n(s)|_0|\m(s)|\d
s\nonumber\\&
 \quad +2\int_0^t\int_Z \big|g(\n(s),z)-g(u_\theta(s),z)\big||\ell(\phi_n(s,z)||\m(s)|\lambda(dz)\d
 s.
\end{align}
For $I_2$, apply $2ab \leq \eta a^2+\frac{1}{\eta}b^2$ for the first
term by taking $\eta=M$ and $2ab\leq a^2+b^2$ for second term by
taking $a=\big|g(\n(s),z)-g(u_\theta(s),z)\big|$ and $b=|\m(s)|$ to
obtain,
\begin{align}
I_2 &\leq M\int_0^t
\big|\sigma(s,\n(s)-\sigma(s,u_\theta(s)\big|_{\mathrm{L}_Q}^2\d
s+\frac{1}{M}\int_0^t|\psi_n(s)|_0^2|\m(s)|^2\d s\nonumber\\&\quad
+\int_0^t\int_Z\big(|g(\n(s),z)-g(u_\theta(s),z)|^2+|\m(s)|^2\big)|\ell(\phi_n(s,z))|\lambda(dz)\d
s.
\end{align}
Let us take the third term, apply Young's inequality, by using the
control condition on $\phi$ and then apply Hypothesis (H.3) to
obtain,
\begin{align}
I_2&\leq M\int_0^t
\big|\sigma(s,\n(s)-\sigma(s,u_\theta(s)\big|_{\mathrm{L}_Q}^2\d
s+\frac{1}{M}\int_0^t|\psi_n(s)|_0^2|\m(s)|^2\d s\nonumber\\&\quad
+\left(\int_0^t\int_Z|g(\n(s),z)-g(u_\theta(s),z)|^2\lambda(dz)\d
s\right)\left(\sup_{0\leq t\leq T}\sup_{z\in
Z}|\ell(\phi(t,z))|\right)\nonumber\\&\quad +\int_0^t\int_Z
|\m(s)|^2|\ell(\phi_n(s,z))|\lambda(dz)\d s\nonumber\end{align}
\begin{align} \leq ML\int_0^t|\m(s)|^2\d s
+\int_0^t\left[\frac{1}{M}|\psi_n(s)|_0^2+\int_Z|\ell(\phi_n(s,z))|\lambda(dz)\right]|\m(s)|^2\d
s.\label{247}
\end{align}
Now let us substitute \eqref{247} in \eqref{24} to obtain,
\begin{align}
&2\left|\int_0^t \big(\sigma(s, \n(s)) \psi_n(s) - \sigma(s,
u_\theta(s)) \psi(s), \m(s)\big)\d s \right|\nonumber\\&
+2\left|\int_0^t\int_Z\big(g(\n(s),z)\ell(\phi_n(s,z))-g(u_\theta(s),z)\ell(\phi(s,z)),\m(s)\big)\lambda(dz)\d s\right|\nonumber\\
&\leq\int_0^t\left[\frac{1}{M}|\psi_n(s)|_0^2+\int_Z|\ell(\phi_n(s,z))|\lambda(dz)+ML\right]|\m(s)|^2\d
s\nonumber\\&\;+2\sup_{0\leq t\leq T}\left|\int_0^t \big(\sigma(s,
u_\theta(s)) (\psi_n(s) -\psi(s)),\m(s)\big)\d
s\right|\nonumber\\&\;+2\sup_{0\leq t\leq
T}\left|\int_0^t\int_Z\big(g(u_\theta(s),z)(\ell(\phi_n(s,z))-\ell(\phi(s,z))),\m(s)\big)\lambda(dz)\d
s\right|.\label{248}
\end{align}
By the boundedness of $\{|\m(s)|^2\}$ in $C(0, T; H)$, and using the
Lemma \ref{BDL1}, the second integral on the right side of
\eqref{248} goes to $0$ as $n \to \infty$. Therefore, given any
$\epsilon > 0$, there exists an integer $N_1$ large so that for all
$n \ge N_1$,
\begin{equation} \sup_{0 \le t \le T}\left|\int_0^t \big(\sigma(s, u_\theta(s)) (\psi_n(s) - \psi(s)),
\m(s)\big)ds\right| < \frac{\epsilon}{4}.\label{244}
\end{equation}
And by applying the dominated convergence theorem, for any given
$\epsilon>0$, there exists an integer $N_2$, large so that for all
$n\geq N_2$,
\begin{align}
\sup_{0\leq t\leq
T}\left|\int_0^t\int_Z\big(g(u_\theta(s),z)(\ell(\phi_n(s,z))-\ell(\phi(s,z))),\m(s)\big)\lambda(dz)\d
s\right|<\frac{\epsilon}{4}.
\end{align}
Choose $N=\max(N_1,N_2)$ so that,
\begin{align}
&\sup_{0 \le t \le T}\left|\int_0^t \big(\sigma(s, u_\theta(s))
(\psi_n(s) - \psi(s)), \m(s)\big)ds\right|\label{249}\\&+\sup_{0\leq
t\leq
T}\left|\int_0^t\int_Z\big(g(u_\theta(s),z)(\ell(\phi_n(s,z))-\ell(\phi(s,z))),\m(s)\big)\lambda(dz)\d
s\right|<\frac{\epsilon}{2}.\nonumber
\end{align}
Let us define $C_{M,L, \nu} =
\max\left\{ML,\frac{1}{M},\frac{1}{\nu},1\right\}.$ Applying
\eqref{249}, \eqref{248} and \eqref{23} in \eqref{22}, one obtains
for $n \ge N$,
\begin{align}
&|\m(t)|^2 + \nu\int_0^t\|\m(s)\|^2 \d s \label{24a}\\&\leq C_{M,L,
\nu}\int_0^t |\m(s)|^2\left(|u_\theta(s)|^2 +
|\psi_n(s)|^2_0+\int_Z|\ell(\phi_n(s,z))|\lambda(dz)+1 \right) \d s
+ \epsilon.\nonumber
\end{align}
From the above relation one can get by denoting $C_{M,L,\nu}$ by
$\mathbb{C}$,
\begin{align*}
&\sup_{0\leq t\leq T} |\m(t)|^2 + \nu\int_0^T\|\m(t)\|^2 \d
t\nonumber\\& \leq \mathbb{C}\int_0^T\sup_{0\leq s \leq T}
|\m(s)|^2\left(|u_\theta(s)|^2 +
|\psi_n(s)|^2_0+\int_Z|\ell(\phi_n(s,z))|\lambda(dz)+1 \right) \d s
+ \epsilon.
\end{align*}
Hence by applying Gronwall's inequality we get,
\begin{align}
&\sup_{0\leq t\leq T} |\m(t)|^2 + \nu\int_0^T\|\m(t)\|^2 \d
t\nonumber\\& \leq \epsilon  \exp\left\{\mathbb{C}\int_0^T
\left(|u_\theta(t)|^2 + |\psi_n(t)|^2_0
+\int_Z|\ell(\phi_n(s,z))|\lambda(dz)+ 1\right)\d
t\right\}.\label{25}
\end{align}
The arbitrariness of $\epsilon$ finishes the proof.
\end{proof}

\begin{remark}
From Theorem \ref{scontrol} one can see that the  equation
\begin{align}
&\d \uve(t) + \left[\nu A\uve(t) + B(\uve(t), \uve(t))\right] \d t\nonumber\\
&\quad = \left[f(t)+\sigma(t, \uve(t))\psi^{\e}(t)+\int_Z
g(\uve(t),z)\ell(\phi^\e(t,z))\lambda(dz)\right]\d t
\nonumber\\&\quad\quad+ \sqrt{\e}\sigma(t, \uve(t))\d W(t)+\e
\int_Zg(\uve(t-),z)\tilde{N}(dt,dz),\label{26}
\end{align}
with $\uve(0) = \xi\in H$, has unique strong solution in
$\mathrm{L}^2(\Omega; X)$.

As we have noted before, the solution of the above equation admits a
representation $\uve = \mathcal{G}^{\e}\left(\sqrt{\e}W(\cdot) +
\int_0^{\cdot}\psi^{\e}(s) \d s,\e N^{\e^{-1}\phi^{\e}}\right)$ by
pathwise uniqueness of the solution, and the Girsanov theorem. For
similar type of formulation readers can refer to Budhiraja, Dupuis
and Maroulas \cite{BD2} Section 4.1, where the authors have
considered small noise stochastic differential equations(SDE) with
finite dimensional jump diffusions.

For all $\theta \in \mathcal{U}$, let $u_\theta$ be the solution of
the deterministic control equation
\begin{align*}
\d u_\theta(t) + [\nu Au_\theta(t)+ B(u_\theta(t), u_\theta(t))]\d t
&= f(t)\d t + \sigma(t, u_\theta(t)) \psi(t) \d
t\nonumber\\&\quad+\int_Zg(u_\theta(t),z)\ell(\phi(s,z))\lambda(dz)\d
s,
\end{align*}
with initial condition $u_\theta(0)=\xi\in H$.

Note that $\int_0^{\cdot}\, \psi(s)ds \in C([0, T]; H_0)$ and
$\int_0^{\cdot}\int_Z\phi(s,z)\lambda(dz)\d s\in C([0,T];H)$. Define
$\mathcal{G}^0: C([0, T]; H_0)\times C([0,T];H) \to X$ by
$$\mathcal{G}^0(h)=u_\theta\quad \mathrm{if}\; h=\left(\int_0^\cdot
\psi(s)\d s,\int_0^\cdot \int_Z \phi(s,z)\lambda(dz)\d s\right)$$
for some $\theta=(\psi,\phi)\in \mathcal{U}$. If $h$ cannot be
represented as above, then $\mathcal{G}^0(h)=0$.
\end{remark}

\begin{theorem}[Weak convergence]
For proving weak convergence, let us define the set to be
$\left\{\theta^{\e} =(\psi^\e,\phi^\e)\in\mathcal{U}:
\theta^\e(\omega)\in
\bar{S}^M\;\mathbb{P}-\;\mathrm{a.e}\;\omega\;\e>0\right\}\subset\mathcal{U}^M$
converges in distribution to $\theta$ with respect to the weak
topology defined on $\mathcal{U}$. Then we have $
\mathcal{G}^{\e}\left(\sqrt{\e}W(\cdot) + \int_0^{\cdot}\psi^{\e}(s)
\d s,\e N^{\e^{-1}\phi^\e}\right)$ converges in distribution to
$\mathcal{G}^0(\int_0^{\cdot}\psi(s) \d s,\lambda_T^\phi)$ in $X$,
as $\e\rightarrow 0$.
\end{theorem}

\begin{proof}
Since $\bar{S}^M$ is a Polish space, the Skorokhod representation
theorem can be introduced to construct processes
$(\tilde{\theta}^{\e}, \tilde{\theta},
\tilde{W}^{\e},\tilde{\lambda}_T)$ such that the distribution of
$(\tilde{\theta}^{\e}, \tilde{\theta},
\tilde{W}^{\e},\tilde{\lambda}_T)$ is same as that of $(\theta^{\e},
\theta, W,\lambda_T)$, and $\tilde{\theta}^{\e} \to \tilde{\theta}$
a.s. in the weak topology of $\bar{S}^M$. Thus $\int_0^t
\tilde{\theta}^{\e}(s) \d s\rightarrow \int_0^t \tilde{\theta}(s) \d
s$ weakly in $H$ a.s. for all $t\in [0, T]$. Without any loss of
generality, we will write $(\theta^{\e}, \theta, W,\lambda_T)$ in
what follows, though strictly speaking, one should write
$(\tilde{\theta}^{\e}, \tilde{\theta},
\tilde{W}^{\e},\tilde{\lambda}_T)$.

Let $\w(t) = \uve(t)- u_\theta(t)$. We need to prove, in probability
as $\e \to 0$, $\sup_{0\leq t\leq T} |\w(t)|^2 + \int_0^T
\|\w(t)\|^2 \d t \to 0.$

For $\w(t)$ we will get the stochastic differential equation as
\begin{align}
&\d \w(t)+[\nu
A\w(t)+B(\uve(t),\uve(t))-B(u_\theta(t),u_\theta(t))]\d t
\nonumber\\&=[\sigma(t,\uve(t))\psi^\e(t)-\sigma(t,u_\theta(t))\psi(t)]\d
t\nonumber\\&\quad+\int_Z\left[g(t,\uve(t))\ell(\phi^\e(t,z))-g(u_\theta(t),z)\ell(\phi(t,z))\right]\lambda(dz)\d
t \nonumber\\&\quad+\sqrt{\e}\sigma(t,\uve(t))\d W(t)+\e
\int_Zg(\uve(t-),z)\tilde{N}(dt,dz).
\end{align}
By applying It\^{o}'s Lemma for the process $|\w(t)|^2$ and
integrating from $0\leq s\leq t$,
\begin{align}
&|\w(t)|^2+2\nu\int_0^t\|\w(s)\|^2\d
s+2\int_0^t(B(\uve(s))-B(u_\theta(s)),\w(s))\d
s\nonumber\\&\quad=I_3+I_4
+2\sqrt{\e}\int_0^t(\sigma(s,\uve(s)),\w(s))\d
W(s)\nonumber\\&\quad\quad+2\e\int_0^t\int_Z(\w(s-),g(\uve(s-),z))\tilde{N}(ds,dz),\;\;\textrm{where}\label{144}
\end{align}
\begin{align*}
&I_3=
2\int_0^t\big(\sigma(s,\uve(s))\psi^\e(s)-\sigma(s,u_\theta(s))\psi(s),\w(s)\big)\d
s\nonumber\\&\quad\quad
+2\int_0^t\int_Z\big(g(\uve(s),z)\ell(\phi^\e(s,z))-g(u_\theta(s),z)\ell(\phi(s,z)),\w(s)\big)\lambda(dz)\d
s,\nonumber\\& I_4= \e \int_0^t
Tr(\sigma(s,\uve(s))Q\sigma(s,\uve(s)))\d s +\e
\int_0^t\int_Z|g(\uve(s-),z)|^2\lambda(dz)\d s.
\end{align*}
Notice that by applying similar techniques as in Theorem
\ref{compact}, one obtains,
\begin{align}
&|I_3|\leq
C_{M,L,\nu}\int_0^t|\w(s)|^2\left(1+|\psi^\e(s)|_0^2+\int_Z|\ell(\phi^\e(s,z))|\lambda(dz)\right)\d
s\nonumber\\&\quad\quad
+2\left|\int_0^t\big(\sigma(s,u_\theta(s))(\psi^\e(s)-\psi(s)),\w(s)\big)\d
s\right|\nonumber\\&\quad \quad+2\left|\int_0^t\int_Z
\big(g(u_\theta(s),z)(\ell(\phi^\e(s,z))-\ell(\phi(s,z))),\w(s)\big)\lambda(dz)\d
s\right|\nonumber\end{align} \begin{align} &\leq
C_{M,L,\nu}\int_0^t|\w(s)|^2\left(1+|\psi^\e(s)|_0^2+\int_Z|\ell(\phi^\e(s,z))|\lambda(dz)\right)\d
s\nonumber\\&\quad
+\int_0^t\left(|\sigma(s,u_\theta(s))(\psi^\e(s)-\psi(s))|^2+|\w(s)|^2\right)\d
s\nonumber\\&\quad
+\int_0^t\int_Z\left(|g(u_\theta(s),z)(\ell(\phi^\e(s,z))-\ell(\phi(s,z)))|^2+|\w(s)|^2\right)\lambda(dz)\d
s\nonumber\\&\leq
C_{M,L,\nu}\int_0^t|\w(s)|^2\left(2+\int_Z\lambda(dz)+|\psi^\e(s)|_0^2+\int_Z|\ell(\phi^\e(s,z))|\lambda(dz)\right)\d
s\nonumber\\&\quad+\int_0^t|\sigma(s,u_\theta(s))|^2|\psi^\e(s)-\psi(s)|^2\d
s\nonumber\\&\quad+\int_0^t\int_Z|g(u_\theta(s),z)|^2|\ell(\phi^\e(s,z))-\ell(\phi(s,z))|^2\lambda(dz)\d
s,
\end{align} where $C_{M,L,\nu}=\max\{ML,\frac{1}{M},\frac{1}{\nu},1\}$.
Now let us take the term $I_4$ from \eqref{144} and apply condition
(H.2) and Corollary \ref{re11} to obtain,
\begin{align}
&\e\int_0^t Tr(\sigma(s,\uve(s))Q\sigma(s,\uve(s)))\d s+\e \int_0^t
\int_Z |g(\uve(s),z)|^2\lambda(dz)\d s\nonumber\\&\quad\leq \e
K\left(T+C\right).
\end{align}
By using the above estimates and denoting $C_{M,L,\nu}$ as
$\mathbb{C}$, $\ell(\phi^\e(t,z)$ as $\ell(\phi^\e)$ and
$\ell(\phi(t,z)$ as $\ell(\phi)$, we can obtain from equation
\eqref{144} by taking supremum from $0\leq t\leq T$ and then
expectation as before,
\begin{align}\label{350}
&\mathbb{E}\left[\sup_{0\leq t\leq T}|\w(t)|^2 + \nu\int_0^T
\|\w(t)\|^2 \d t \right]\nonumber\\& \leq
\mathbb{C}\mathbb{E}\left[\int_0^T\sup_{0\leq t\leq
T}|\w(t)|^2\left(2+|u_\theta(t)|^2+|\psi^\e(t)|_0^2+\int_Z(1+|\ell(\phi^\e)|)\lambda(dz)\right)\d
t\right]\nonumber\\&\quad+\e
K\left(T+C\right)+\int_0^T|\sigma(t,u_\theta(t))|^2|\psi^\e(t)-\psi(t)|^2\d
t\nonumber\\&\quad+\int_0^T\int_Z|g(u_\theta(t),z)|^2|\ell(\phi^\e)-\ell(\phi)|^2\lambda(dz)\d
t\nonumber\\&\quad+2\sqrt{\e}\mathbb{E}\left[\sup_{0\leq t\leq
T}\left|\int_0^t\big(\sigma(s,\uve(s)),\w(s)\big)\d
W(s)\right|\right]\nonumber\\&\quad+2\e\mathbb{E}\left[\sup_{0\leq
t\leq
T}\left|\int_0^t\int_Z\big(\w(s-),g(\uve(s-),z)\big)\tilde{N}(ds,dz)\right|\right].
\end{align}
Let us take the term $2\sqrt{\e}\mathbb{E}\left[\sup_{0\leq t\leq
T}\left|\int_0^t\big(\sigma(s,\uve(s)),\w(s)\big)\d
W(s)\right|\right]$ from \eqref{350} and apply
Burkholder-Davis-Gundy inequality, Young's inequality, Hypothesis
(H.2) and Corollary \ref{re11} to obtain,
\begin{align}
&2\sqrt{\e}\mathbb{E}\left(\sup_{0\leq t\leq
T}\left|\int_0^t\big(\sigma(s,\uve(s)),\w(s)\big)\d
W(s)\right|\right)\nonumber\\&\leq 2\sqrt{2\e
K}\mathbb{E}\left(\int_0^T(1+|\uve(s)|^2)|\w(s)|^2\d s\right)^{1/2}
\nonumber\\&\leq 2\sqrt{2\e K}\left[\frac{1}{8\sqrt{2\e
K}}\mathbb{E}\left(\sup_{0\leq t\leq T}|\w(t)|^2\right)+2\sqrt{2\e
K}\mathbb{E}\left(\int_0^T(1+|\uve(s)|^2)\d
s\right)\right]\nonumber\\&\leq
\frac{1}{4}\mathbb{E}\left(\sup_{0\leq t\leq T}|\w(t)|^2\right)+8\e
K(T+C)\label{250}.
\end{align}
Consider the term $2\e\mathbb{E}\left[\sup_{0\leq t\leq
T}\left|\int_0^t\int_Z\big(\w(s-),g(\uve(s-),z)\big)\tilde{N}(ds,dz)\right|\right]$
from \eqref{350} and apply the Burkholder-Davis-Gundy inequality in
the form given in Lemma \ref{bdg}, Hypothesis (H.2) and Corollary
\ref{re11} to obtain,
\begin{align}
&2\e \mathbb{E}\left(\sup_{0\leq t\leq
T}\left|\int_0^t\int_Z\big(\w(s-),g(\uve(s-),z)\big)\tilde{N}(ds,dz)\right|\right)
\nonumber\\&\leq
2\e\sqrt{2}\mathbb{E}\left(\int_0^T\int_Z|\w(s)|^2|g(\uve(s),z)|^2\lambda(dz)\d
s\right)^{1/2}\nonumber\\&\leq 2\e
\sqrt{2K}\mathbb{E}\left[\sup_{0\leq t\leq
T}|\w(t)|\left(\int_0^T(1+|\uve(s)|^2)\d
s\right)^{1/2}\right]\nonumber\\&\leq 2\e
\sqrt{2K}\left[\frac{1}{8\e\sqrt{2K}}\mathbb{E}\left(\sup_{0\leq
t\leq
T}|\w(t)|^2\right)+2\e\sqrt{2K}\mathbb{E}\left(\int_0^T(1+|\uve(s)|^2)\d
s\right)\right]\nonumber\\&\leq
\frac{1}{4}\mathbb{E}\left(\sup_{0\leq t\leq
T}|\w(t)|^2\right)+8K\e^2(T+C).
\end{align}
By using all these estimates, we can reduce the inequality
\eqref{350} as,
\begin{align}
&\mathbb{E}\left[\sup_{0\leq t\leq T}|\w(t)|^2 + 2\nu\int_0^T
\|\w(t)\|^2 \d t\right]\nonumber\\
&\leq 2\mathbb{C}\mathbb{E}\left[\int_0^T \sup_{0\leq t\leq
T}|\w(t)|^2\left(2+|u_\theta(t)|^2 + |\psi^{\e}(t)|_0^2+\int_Z
(1+|\ell(\phi^\e)|)\lambda(dz) \right)\d
t\right]\nonumber\\&\quad+2K\e [(9+\e)(C+T)]+2\int_0^T |\sigma(t,
u_\theta(t))|^2 |(\psi^{\e}(t) - \psi(t))|^2 \d t
\nonumber\\
&\quad+
2\int_0^T\int_Z|g(u_\theta(t),z)|^2|\ell(\phi^\e(t,z))-\ell(\phi(t,z))|^2\lambda(dz)\d
t.
\end{align}
Then the Gronwall's inequality yields,
\begin{align}
&E\left[\sup_{0\leq t\leq
T}|\w(t)|^2 + 2\nu\int_0^T \|\w(t)\|^2 \d t\right] \nonumber\\
&\leq \left(2K\e [(9+\e)(C+T)]++2\int_0^T |\sigma(t, u_\theta(t))|^2
|(\psi^{\e}(t) - \psi(t))|^2 \d
t\right.\nonumber\\&\quad\qquad\quad\left.+
2\int_0^T\int_Z|g(u_\theta(t),z)|^2|\ell(\phi^\e(t,z))-\ell(\phi(t,z))|^2\lambda(dz)\d
t\right)\label{29}\\&\quad\quad \times
\exp\left\{2\mathbb{C}\int_0^T\left(2+|u_\theta(t)|^2 +
|\psi^{\e}(t)|_0^2+\int_Z (1+|\ell(\phi^\e(t,z))|)\lambda(dz)
\right)\d t\right\}.\nonumber
\end{align}
We have given that $\theta^\e(t)\rightarrow\theta(t)$ a.s in the
weak topology of $\mathcal{U}^M$. Since $\psi^{\e} \to \psi$ a.s. in
the weak topology of $\tilde{S}^M$ and
$\ell(\phi^\e(t,z))\rightarrow\ell(\phi(t,z))$ in a.s the weak
topology of $S^M$ (for further details see Theorem 4.4 of
\cite{Mv}), it is clear from the equation \eqref{29} that as $\e \to
0$, $ \mathbb{E}\left[\sup_{0\leq t\leq T}|\w(t)|^2 + 2\nu\int_0^T
\|\w(t)\|^2 \d t\right] \to 0. $ Let $\delta > 0$ be any arbitrary
number. Then by Markov's inequality
\begin{align*}
&\mathbb{P}\left\{\sup_{0\leq t\leq T}|\w(t)|^2 + 2\nu\int_0^T
\|\w(t)\|^2 \d t
\geq \delta\right\}\nonumber\\
&\quad \leq \frac{1}{\delta}\mathbb{E}\left[\sup_{0\leq t\leq
T}|\w(t)|^2 + 2\nu\int_0^T \|\w(t)\|^2 \d t\right] \to 0\ \text{as}\
\e \to 0.
\end{align*}
Thus, $ \sup_{0\leq t\leq T}|\uve(t)-u_\theta(t)|^2 + \nu\int_0^T
\|\uve(t)-u_\theta(t)\|^2 \d t \to 0, $ in probability as $\e\to 0$.
The proof is now complete.
\end{proof}

\medskip\noindent
{\bf Acknowledgements:} Manil T. Mohan would like to thank Council
of Scientific and Industrial Research (CSIR), India for a Senior
Research Fellowship (SRF). The authors would also like to thank
Indian Institute of Science Education and Research (IISER)-
Thiruvananthapuram for providing stimulating scientific environment
and resources. The authors would also like to thank the anonymous
referee for his/her valuable comments.

\end{document}